\documentclass[11pt]{article}
\usepackage{amsmath, amsfonts, amsthm,amssymb, color, hyperref, enumerate, extarrows}

\newtheorem{theorem}{Theorem}
\newtheorem*{theorem*}{Theorem}

\newtheorem{lem}{Lemma}[section]
\newtheorem{cor}{Corollary}
\newtheorem{definition}[lem]{Definition}
\newtheorem*{notation*}{Notation}
{\theoremstyle{definition}
\newtheorem{example}[lem]{Example}}
\newtheorem{pro}[lem]{Proposition}
\newtheorem*{pro*}{Proposition}

\numberwithin{equation}{section}
 {\theoremstyle{definition}
 \newtheorem{remark}[lem]{Remark}}
 \newtheorem*{remark*}{Remark}
 \newtheorem*{claim}{Claim}

\newcounter{nmdthmcnt}

\newenvironment{customtheorem}[1]
{\par\noindent\textbf{Theorem #1.} \itshape}
{\par}

\usepackage[T1]{fontenc}

\hypersetup{
	colorlinks   = true,
	citecolor    = magenta}

\newcommand{\ssubset}{\subset\joinrel\subset}

\newcommand{\bC}{\mathbb{C}}
\newcommand{\bB}{\mathbb{B}}

\begin{document}

\title{\vspace{-1.2cm} \bf Domains with Bergman metrics of constant curvature and Bergman-negligible subsets\rm}

\author{Peter Ebenfelt, John N. Treuer, Ming Xiao}
\date{}

\maketitle

\begin{abstract}
Let $D$ be a bounded domain in $\mathbb{C}^n$.  Suppose the holomorphic sectional curvature of its Bergman metric equals a negative constant $\tau$.  We show that $D$ is biholomorphic to a domain $\Omega$ equal to the unit ball in $\mathbb{C}^n$ less a relatively closed set of measure zero, and that all $L^2$-holomorphic functions on $\Omega$ extend to $L^2$-holomorphic functions on the ball.  Consequently,
$\tau$ must equal the holomorphic sectional curvature of the unit ball.  This generalizes a classical theorem of Lu.  Some applications of the theorem, especially in extending classical work of Wong and Rosay, are also presented.

\end{abstract}

\renewcommand{\thefootnote}{\fnsymbol{footnote}}
\footnotetext{\hspace*{-7mm}
\begin{tabular}{@{}r@{}p{16.5cm}@{}}
& Keywords. Uniformization, Bergman metric, Holomorphic sectional curvature, moment problem, biholomorphism\\
& Mathematics Subject Classification. Primary 32Q05; Secondary 32Q30, 32A36, 32H02, 44A60
%
\end{tabular}

\noindent\thanks{The first author is supported in part by the NSF grant DMS-2154368. The second author is supported in part by the NSF grant DMS-2247175 Subaward M2401689. The third author is
supported in part by the NSF grant DMS-2045104.}

\noindent \date
}

\section{Introduction}\label{intro}
The Bergman kernel and metric are among the most fundamental concepts in several complex variables. They provide a wealth of biholomorphic information for the study of geometric problems on complex domains, and more generally, on complex manifolds.
One of the most basic models for such domains and manifolds is the unit ball $\mathbb{B}^n$ in the complex space $\mathbb{C}^n$. Its Bergman kernel  can be explicitly computed and shown to be
$$
K_{\mathbb{B}^n}(z, w) = \frac{n!}{\pi^n}\frac{1}{(1-\left<z, w\right>)^{n+1}},$$
where $\langle z, w \rangle = \sum_{j=1}^n z_j\overline{w}_j$.
We will mostly be concerned with the on-diagonal Bergman kernel denoted as $K := K(z, z)$.  It is also routine to verify that the Bergman metric,
\begin{equation*}
	(g_{\mathbb{B}^n})_{i\bar j}=\frac{\partial^2}{\partial{z_i}\partial{\bar z_j}}\log K_{\mathbb{B}^n},
\end{equation*}
has constant holomorphic sectional curvature $-2(n+1)^{-1}$.
 In 1965, Lu \cite{L65} proved his celebrated uniformization theorem of the Bergman metric.
\medskip
\begin{customtheorem}{0}{\rm (Lu's uniformization theorem, \cite{L65}).}
Let $D$ be a bounded domain in $\mathbb{C}^n$ such that its Bergman  metric $g$ is complete.  The Bergman  metric $g$ has constant holomorphic sectional curvature $\tau$ if and only if $D$ is biholomorphic to the unit ball $\mathbb{B}^n$. Moreover, in this case, the constant $\tau$ equals $-2(n+1)^{-1}$.
\end{customtheorem}
\medskip

Lu's uniformization theorem has significant applications in the study of the geometry of complex domains through the Bergman metric. In particular, it plays a key role in proving the well-known Cheng Conjecture, which asserts that a smoothly bounded strongly pseudoconvex domain whose Bergman metric is K\"ahler-Einstein must be biholomorphic to the unit ball.  The reader is referred to Fu--Wong \cite{FW97} and Nemirovski--Shafikov \cite{NS} for the solution of Cheng's Conjecture in $\mathbb{C}^2$, and Huang--Xiao \cite{HX21} for its solution in higher dimensions.

Note that in Lu's Theorem 0, since the Bergman metric $g$ is assumed to be complete with constant holomorphic sectional curvature, the domain $D$ must be covered by one of the space forms by the classical uniformization theorem. But, since $D$ is a bounded domain in $\bC^n$, it can only be covered by the hyperbolic ball ($\mathbb{B}^n$ equipped with its Bergman metric). Consequently, the constant curvature $\tau$ is necessarily negative. A longstanding folklore question asks, in the context of Lu's theorem, if and how the completeness assumption of the Bergman metric can be relaxed.  In particular, if the completeness of the Bergman metric $g$ is not assumed, then the classical uniformization theorem would not apply and the sign of $\tau$ would be unclear. 
The first breakthrough contributions were made by Dong--Wong \cite{DW22b, DW22a} and Huang--Li \cite{HL24}; see also an earlier work by Huang-Xiao \cite{HX20}.
Huang--Li \cite{HL24} constructed a domain whose Bergman metric has {\em positive} constant holomorphic sectional curvature. They proved, however, that such domains cannot be bounded by demonstrating that the Bergman space must be finite dimensional in this case. Together with Treuer, Huang--Li \cite{HL24} also showed that there are no bounded domains in  $\mathbb{C}^n$ whose Bergman metric has zero holomorphic sectional curvature. As a result,  if the Bergman metric of a bounded domain has constant holomorphic sectional curvature $\tau$, then $\tau$ must be negative.  Huang--Li \cite{HL24} further proved:

\medskip

\begin{customtheorem}{A}
{\rm (Huang-Li, \cite{HL24}).} {\it Let $D$ be a bounded pseudoconvex domain in $\mathbb{C}^n.$
	 The Bergman  metric $g$ of $D$ has constant holomorphic sectional curvature $\tau$ if and only if
$D$ is biholomorphic to a domain $\Omega \subset \mathbb{C}^n$, where $\Omega$ is the unit ball $\mathbb{B}^n$ less possibly a closed pluripolar subset.}
\end{customtheorem}

\medskip

Recall that a set $E$ is called pluripolar if it is locally a subset of the -$\infty$ level set of a plurisubharmonic function at every $p \in E$. By Kobayashi \cite{K59}, a bounded domain in $\mathbb{C}^n$ with a complete Bergman metric is necessarily pseudoconvex. On the other hand, however, a bounded pseudoconvex domain may not have a complete Bergman metric. For instance, the Bergman metric of the pseudoconvex domain consisting of a complex ball with a complex hyperplane deleted is not complete, as it is bounded across the deleted hyperplane. Therefore, assuming pseudoconvexity, as in Theorem A, is strictly weaker than assuming a complete Bergman metric, as in Lu's Theorem 0.
Huang-Li's theorem gives the optimal extension of Lu's theorem to bounded pseudoconvex domains.

\begin{remark}\label{general Huang Li}
Huang--Li \cite{HL24} considered a more general setting than that of a bounded domain. They worked on a complex manifold $M$ satisfying the following two conditions: 
\begin{enumerate}[(a)]
    \item  $M$ admits a well-defined Bergman metric, meaning that its Bergman kernel $K_M$ is nowhere vanishing and $i\partial\overline{\partial} \log K_M$ is a positive-definite Hermitian form on $M$.
    \item The Bergman space separates points, meaning that for any two distinct points in $M$, there exists a Bergman space element which equals zero at one point and is nonzero at the other.
\end{enumerate}
Conditions (a) and (b) are always fulfilled if $M$ is a bounded domain in $\mathbb{C}^n.$ Among other things, Huang--Li \cite{HL24} established a more general version of Theorem A in which the pseudoconvex domain $D$ is replaced by any Stein manifold satisfying conditions (a) and (b).
\end{remark}

In this paper, we consider the problem of characterizing bounded domains without any pseudoconvexity assumption.
Our main result is as follows.

\begin{theorem}\label{main theorem - negative constant case}
Let $D$ be a bounded domain in $\mathbb{C}^n$.  Then, the following are equivalent:
\begin{enumerate}[(1)]
    \item[{\rm (1)}]\label{case 1} The Bergman metric $g$ of $D$ has constant holomorphic sectional curvature $\tau$;
    \item[{\rm (2)}]\label{case 2} $D$ is biholomorphic to a domain $\Omega \subset \mathbb{B}^n$, where $E := \mathbb{B}^n\setminus\Omega$ is a set of Lebesgue measure zero (thus $E \subset \partial \Omega$), and every $L^2$-holomorphic function on $\Omega$ extends to an $L^2$-holomorphic function on $\mathbb{B}^n$.
\end{enumerate}
Moreover, in either case, the constant $\tau = -{2(n + 1)^{-1}}$ and the Bergman kernel of $\Omega$ is the restriction of that of $\mathbb{B}^n$ to $\Omega$.
\end{theorem}

Theorem \ref{main theorem - negative constant case} gives the final extension of Lu's uniformization theorem for bounded domains. Note that, by the discussion preceding Theorem A, the constant $\tau$ in \eqref{case 1} is {\em a priori} necessarily negative. Theorem \ref{main theorem - negative constant case} will be a special case of a more general theorem, Theorem \ref{general version main theorem} below, where the bounded domain hypothesis is weakened to complex manifolds satisfying conditions (a) and (b) in Remark \ref{general Huang Li}.  The proof of Theorem \ref{general version main theorem} will build on a key step of Huang-Li \cite{HL24}, stated as Theorem \ref{The portion of Huang and Li's Theorem we need} below, used to establish Theorem A and its more general formulation.


\begin{remark}\label{rmktothm} We make a few more remarks about Theorem \ref{main theorem - negative constant case}.
    \begin{enumerate}

   \item When $n = 1$, the conditions on $E$ and $\Omega$ in (2) hold if and only if the set $E$ is a closed polar subset of the disk. This follows directly from \cite[9.9 Theorem, p.~351]{Conway95}. Also note that when $n=1$, every domain is pseudoconvex and, hence, Theorem \ref{main theorem - negative constant case} reduces to Theorem A.


        \item When $n \geq  2$, the set $E$ in (2) may not be a pluripolar set.  For example, let  $\epsilon > 0$ be small and set
        $$E = \left\{(z', z_n ) \in \mathbb{C}^{n-1} \times \mathbb{C}: |z'|^2+|\mathrm{Re} z_n|^2 \leq \epsilon, \mathrm{Im} z_n =0 \right\} \ssubset \mathbb{B}^n$$ and $ \Omega = \mathbb{B}^n\setminus E$.
        Note that $E$ is not pluripolar because it has Hausdorff codimension one, and pluripolar sets always have Hausdorff codimension at least two. On the other hand, it is also easy to verify that $\Omega$ satisfies the conditions in (2) by using Hartogs' extension theorem.

        \item When $n \geq 2,$ if in Theorem \ref{main theorem - negative constant case} we additionally assume that $D$ (and, hence, $\Omega$) is pseudoconvex, then $E$ in (2) must be pluripolar. Indeed, since the Bergman kernel $K_{\Omega}=K_{\mathbb{B}^n}$, it holds that $\lim_{\Omega \ni z \to w} K_{\Omega}(z) < \infty$ for every $w \in E$. Then, as observed in Huang-Li \cite{HL24}, by Pflug and Zwonek \cite[Lemma 11]{PZ02}, $E$ is pluripolar. Therefore, if one adds the assumption of pseudoconvexity of $D$ to Theorem \ref{main theorem - negative constant case}, then the result reduces to Theorem A.

        \item When $n \geq 2,$ even if $E$ is a pluripolar set, $\Omega = \mathbb{B}^n \setminus E$ may not be a pseudoconvex domain.  For instance, take $E = \{0\}$. Then, $E$ is pluripolar, but $\Omega $ is not pseudoconvex.


    \end{enumerate}
\end{remark}

Motivated by Theorem \ref{main theorem - negative constant case}, we make the following definition.

\begin{definition}
Let $U$ be a domain in $\mathbb{C}^n$. A closed subset $F \subset U$ is called a Bergman-negligible subset of $U$ if $F$ is of Lebesgue measure zero, and every $L^2$-holomorphic function on $U \setminus F$ extends as an $L^2$-holomorphic function on $U$. We allow $F$ to be empty.
\end{definition}

It is evident that a bounded domain $U \subset \mathbb{C}^n$ less a Bergman-negligible subset is necessarily connected (and hence remains a domain). For $n=1,$ a closed subset $F \subset U$ is Bergman-negligible if and only if it is polar (again this follows from \cite[9.9 Theorem, p.~351]{Conway95}).
In higher dimensions, every pluripolar subset $F$ has zero Lebesgue measure and $L^2$-holomorphic functions on $U \setminus F$ extend as $L^2$-holomorphic functions on $U$ (cf. \cite[Lemma 1]{Irgens04} for a proof of this fact). Consequently, a closed pluripolar subset of $U$ is Bergman-negligible.
However, as in part 2 of Remark \ref{rmktothm}, a Bergman-negligible subset is not necessarily pluripolar.
We further observe the following fact about Bergman-negligible subsets, which justifies the terminology. A proof can be found at the end of \S \ref{subsection finishing proof}.

\begin{pro}\label{pronegligibleset}
Let $\hat{U}$ and $U$ be two bounded domains in $\mathbb{C}^n$ with $\hat{U} \subset U.$ Then their Bergman kernels $K_{\hat{U}}$ and $K_U$ are equal on $\hat{U}$ if and only if $U \setminus \hat{U}$ is a Bergman-negligible subset of $U$.
\end{pro}


Proposition \ref{pronegligibleset} implies that if $F$ is a  Bergman-negligible subset of a bounded domain $U$, then for every $w \in F$,  the limit $\lim_{\Omega \ni z \to w} K_{U\setminus F}(z)$ is finite. Thus, by applying Pflug and Zwonek \cite[Lemma 11]{PZ02} as in part 3 in Remark \ref{rmktothm}, we see the following holds: For a bounded domain $U$ and a closed subset $F$ of $U$ such that $U \setminus F$ is pseudoconvex, $F$ is Bergman-negligible if and only if it is  pluripolar.

Next, we discuss applications of our main theorem in extending the well-known classical work
of Wong \cite{Wo77} and Rosay \cite{Ro79}.  Let $\mathrm{Aut}(D)$ denote the group of (biholomorphic) automorphisms of a domain $D$. A point $q \in \partial D$ is called a boundary orbit accumulation point (of $D$) if there is a sequence $\{\phi_j\}_{j=1}^{\infty} \subset \mathrm{Aut}(D)$ and $p \in \Omega$ such that $\phi_j(p) \to q$ as $j \to \infty$. The following results are
due to Wong \cite{Wo77} and Rosay \cite{Ro79}.

\medskip

\begin{customtheorem}{B} {\rm (Wong \cite{Wo77} and Rosay \cite{Ro79}).} Let $D \subset \mathbb{C}^n$ be a bounded domain.
\begin{enumerate}[(1)]
\item[{\rm (1)}]\label{case pnwr1}
If $D$ is homogeneous and has $C^2$-smooth boundary, then $D$ is biholomorphic to the unit ball $\mathbb{B}^n$.
\item[{\rm (2)}]\label{case pnwr2} If there is a boundary orbit accumulation point $q \in \partial D$ such that $D$ has $C^2$-smooth and strongly pseudoconvex boundary at $q$,  
then $D$ is biholomorphic to the unit ball $\mathbb{B}^n$.
\end{enumerate}
\end{customtheorem}

We note that (2) implies (1), as every bounded domain with $C^2$-smooth boundary must have a strongly pseudoconvex boundary point. We also remark that under the assumption of (1) or (2), the Bergman metric $g_D$ of $D$ must have constant holomorphic sectional curvature. One can deduce this by applying the work of Kim and Yu \cite{KiYu96} (without using the conclusion of Theorem B). For a detailed proof of this, see \cite[page 161]{KiYu96}. We point out that our notion of holomorphic sectional curvature differs from that in \cite{KiYu96} by a multiple of two. This is due to the different normalization conventions of the K\"ahler metric associated to a given K\"ahler form.  

Inspired by the above results and discussion, we shall deduce a couple of corollaries from our main theorem that generalize Theorem B in various ways. We shall need the following definition.

\begin{definition}\label{defnc0bdry}
A domain $D \subset \mathbb{C}^n$ has $C^0$-boundary at a point $q \in \partial D$, if there exists a homeomorphism from some neighborhood $O \subset \mathbb{C}^n$ of $q$ to a domain $O_1 \times O_2 \subset \mathbb{R}^{2n-1} \times \mathbb{R}$, which maps $O \cap D$  onto $\{(x,y) \in O_1 \times O_2: y < \phi(x) \}$ for some continuous function $\phi$ on $O_1$. We further say $D$ has $C^0$-boundary, if it has $C^0$-boundary at every $q \in \partial D.$
\end{definition}

Our Corollary \ref{cor2} concerns part (1) of Theorem B, and Corollary \ref{cor3} pertains to part (2).

\begin{cor}\label{cor2}
Let $D \subset \mathbb{C}^n$ be a bounded domain with $C^{0}$-boundary.  If its Bergman metric has constant holomorphic sectional curvature, then  $D$ is biholomorphic to $\mathbb{B}^n$.
\end{cor}

Let $g_D$ be the Bergman metric of a domain $D$ and let
$H(g_D)(z, X)$ denote the holomorphic sectional curvature of $g_D$ at $z \in D$ of a vector $X \in T_zD\setminus\{0\}$ (see Kobayashi \cite{K59} for the precise definition).


\begin{cor}\label{cor3}
Let $D \subset \mathbb{C}^n$ be a bounded domain with Bergman metric $g_D$. Suppose that there is a boundary orbit accumulation point $q \in \partial D$ and $\tau \in \mathbb{R}$ such that
\begin{equation}\label{eqnhsc}
\lim_{z \to q} \sup_{X \in T_z D  \setminus\{0\}} \left| H(g_D)(z, X) - \tau \right| = 0.
\end{equation}
Moreover, suppose that there exists a neighborhood $O$ of $q$ in $\mathbb{C}^n$ such that $\partial D \cap O$ contains no (nontrivial) complex varieties.
Then, $\tau = -2(n + 1)^{-1}$ and $D$ is biholomorphic to $\mathbb{B}^n$ less a Bergman-negligible subset. If, in addition, $D$ has $C^0$-boundary at $q$, then $D$ is biholomorphic to $\mathbb{B}^n$.
\end{cor}

We remark that in Corollary \ref{cor3}, if $D$ has $C^2-$smooth and strongly pseudoconvex boundary at $q$, then (\ref{eqnhsc}) holds  with  $\tau = -2(n + 1)^{-1}$ by Kim and Yu \cite{KiYu96} (see also related results in \cite{Di70} and \cite{Kl78}). Thus, Corollary \ref{cor3} implies part (2) of Theorem B.

\begin{remark}\label{rmkcor23} Some further remarks are in order.
	\begin{enumerate}
		\item Corollary \ref{cor2} is optimal in the following sense. There exists a bounded domain $D$ such that its Bergman metric has constant holomorphic sectional curvature, $D$ has $C^0$-boundary at all points except one on $\partial D$, but $D$ is not biholomorphic to $\mathbb{B}^n$. A trivial example of such a domain is $\mathbb{B}^n \setminus \{0\}.$ A more interesting domain of this kind is constructed in Example \ref{example 2 for corollary 23}.
		
		\item Corollary \ref{cor3} is also optimal: There is a bounded domain which satisfies all conditions in Corollary \ref{cor3}, except that it fails to have $C^0$-boundary at the boundary orbit accumulation point $q$ (while it has $C^0$-boundary at all other points). Furthermore, the domain is not biholomorphic to $\mathbb{B}^n$. See Example \ref{example 2 for corollary 23}.
	\end{enumerate}
\end{remark}

We conclude the introduction by discussing an additional application of our main theorem. Let $D$ be as in Theorem \ref{main theorem - negative constant case}.  Inspired by \cite[Theorem 2]{L65}, \cite[Theorem 1.2]{DW22a} and \cite[Theorem 1.5]{DW22b}, we  shall apply Theorem \ref{main theorem - negative constant case} to obtain an explicit formula in terms of the Bergman kernel for a biholomorphism from $D$ to a complex ellipsoid less a Bergman-negligible subset. Here, by complex ellipsoid, we mean, for some $n \times n$ positive definite Hermitian matrix $H = (h_{i\bar{j}})_{i, \bar{j} = 1}^n$, the domain defined by
\begin{equation}\label{Ellipsoid equation}
	E_H = \left \{\zeta = (\zeta_1,\ldots, \zeta_n) \in \mathbb{C}^n:\, \sum_{i, j = 1}^n \zeta_ih_{i\bar{j}}\bar{\zeta}_j < n + 1\right \}.
\end{equation}
The constant $n+1$ is simply for convenience, of course. Moreover, a complex ellipsoid can be subsequently holomorphically transformed to $\mathbb{B}^n$  by an affine linear map. To state our result, we recall the Bergman representative coordinates.  Let $D$ be a bounded domain in $\mathbb{C}^n$ with Bergman kernel $K = K_{D}$ and Bergman metric $g = (g_{i\bar{j}})_{i, \bar{j} = 1}^n$.  For any $p \in D$, the Bergman representative coordinates of $D$ associated to $p$ are given by $w = (w_1,\ldots, w_n)$ where
\begin{equation}\label{Bergman representative coordinates}
	w_i(z) = T_{p, i}(z) := \sum_{j=1}^n g^{i \bar{j}}(p)\left(\left.{\partial \over \partial \bar{\zeta}_j}\right|_{\zeta = p}\log K(z, \zeta) - \left.{\partial \over \partial \bar{\zeta}_j}\right|_{\zeta = p} \log K(\zeta, \zeta)\right), \quad 1 \leq i \leq n.
\end{equation}
Here $(g^{i \bar{j}})_{i, \bar{j} = 1}^n$ is the inverse of $(g_{i\bar{j}})_{i, \bar{j} = 1}^n$.  As is well-known, $T_p(z) := (T_{p, 1}(z),\ldots, T_{p, n}(z))$ is holomorphic on $D$ away from $\{z \in D:\, K(z, p) = 0\}$ and maps $p$ to 0. We write $E_{g(p)}$ for the ellipsoid as in \eqref{Ellipsoid equation} with $H=g(p)$. We shall prove (in Section \ref{sectioncor1}):

\begin{cor}\label{cor1ellipsoid}
Let $D \subset \mathbb{C}^n$ be a bounded domain with $p \in D$. Assume its Bergman metric $g = (g_{i\bar{j}})_{i, \bar{j} = 1}^n$ has constant holomorphic sectional curvature equal to $\tau$.
Then $\tau = -2(n + 1)^{-1}$. Moreover,
the map $T_p(z)$ in $\eqref{Bergman representative coordinates}$ gives a biholomorphism from $D$ to $E_{g(p)}$ less a Bergman-negligible subset.

\end{cor}

The paper is organized as follows. Section \ref{Preliminaries section} provides background on the moment problem along with the necessary preparations for proving Theorem \ref{main theorem - negative constant case}. In particular, our proof relies on certain parts of the work of Huang--Li \cite{HL24}, and we shall isolate these parts of their work in \S \ref{The part of their proof that we need}.  We then prove Theorem \ref{main theorem - negative constant case} in Section \ref{Section proof of theorem}. The proof is divided into three steps, presented in \S \ref{subsection moment problem theorems}, \S \ref{subsection equality measure zero}, and \S \ref{subsection finishing proof}, respectively. We establish Corollary \ref{cor2} and \ref{cor3} in Section \ref{sectioncor23} and give a few examples to justify Remark \ref{rmkcor23} in Section \ref{sectionexamples}. We conclude the paper with a proof of Corollary \ref{cor1ellipsoid} in Section \ref{sectioncor1}.

\medskip

\textbf{Acknowledgements: } The second author would like to thank Xiaojun Huang and Song-Ying Li for encouragement and for suggesting problems related to the Bergman metric to study. The authors would also like to thank Xiaojun Huang for many inspiring conversations during his visit to UCSD in the Winter Quarter of 2024.

\section{Preliminaries}\label{Preliminaries section}

\makeatletter
\renewcommand{\thetheorem}{\thesection.\arabic{theorem}}
\@addtoreset{theorem}{subsection}
\makeatother

\subsection{The moment problem}\label{moment problem subsection}

Our approach to the main theorem, Theorem \ref{main theorem - negative constant case}, is different from Huang-Li's proof of Theorem A (the negative holomorphic sectional curvature case in \cite{HL24}). We will utilize the (complex) moment problem. This was also employed in Huang-Li-Treuer's treatment of the identically zero holomorphic sectional curvature case (see the appendix of \cite{HL24}). However, unlike their treatment, the key idea here is to integrate the moment problem with the unitary action on the domain.

The moment problem asks, given a sequence of numbers $(s_{\alpha\beta})_{\alpha, \beta \in \mathbb{N}^n}$ called the moment data, does there exist a unique nonnegative Borel measure $\eta$ such that
\begin{equation}\label{definition of complex moment problem}
\int_{\mathbb{C}^n} z^{\alpha}\overline{z}^{\beta}\, d\eta = s_{\alpha\beta}, \quad ~\text{for all}~~ \alpha, \beta \in \mathbb{N}^n?
\end{equation}
The proof of Theorem \ref{main theorem - negative constant case} will lead naturally to a moment problem, which will have a solution $\eta$ of the form $|f|^2\chi_{\Omega} m_{2n}$ where $\chi_{\Omega}$ is the characteristic function of a domain $\Omega$, $f \in A^2(\Omega)$, and $m_{2n}$ denotes the Lebesgue measure on $\mathbb{R}^{2n}$.  All domains $\Omega$ under consideration will be contained in $\mathbb{B}^n \subset \mathbb{C}^n$.  If a measure $\eta$ solving \eqref{definition of complex moment problem} has compact support, then by an application of Weierstrass' approximation theorem it is unique \cite[Proposition 12.17]{Sc17}. For more on the moment problem we refer the reader to \cite{Sc17}.
\subsection{Huang and Li's work}\label{The part of their proof that we need}

In Huang-Li \cite[Theorem 4.1]{HL24}, the following is established:

\begin{theorem}[Huang-Li, \cite{HL24}]\label{The portion of Huang and Li's Theorem we need}   Let $M$ be a connected complex manifold satisfying the two conditions (a) and (b) in Remark \ref{general Huang Li}.  Suppose that the Bergman metric of $M$ has constant negative holomorphic sectional curvature.  Then
$M$ is biholomorphic to a domain $\Omega \subset \mathbb{B}^n$ containing zero and
\begin{equation}\label{Huang and Li's formula of the Bergman kernel}
	K_{\Omega}(z, w) = \phi(z)\overline{\phi(w)}K^{\lambda}_{\mathbb{B}^n}(z, w), \quad z, w \in \Omega,
\end{equation}
where $\phi \in A^2(\Omega)$ with $\phi(0) > 0$, and $\lambda$ is some positive constant.
	


\end{theorem}

In our proof of Theorem \ref{main theorem - negative constant case}, the starting point will be Theorem \ref{The portion of Huang and Li's Theorem we need}.  Specifically we will begin by using \eqref{Huang and Li's formula of the Bergman kernel} to set up a complex moment problem and use the moment data to infer the desired geometric properties of $\Omega$.  This is the essence of Section \ref{subsection moment problem theorems} below.  


\section{Proof of the main theorem }\label{Section proof of theorem}

As was mentioned in the introduction, by using Huang--Li's work, Theorem \ref{The portion of Huang and Li's Theorem we need}, we can prove a more general version of Theorem \ref{main theorem - negative constant case}, which we state as follows:

\begin{theorem}\label{general version main theorem}
Let $M$ be a connected complex manifold satisfying the two conditions (a) and (b) in Remark \ref{general Huang Li}. 
Then the following are equivalent:
\begin{enumerate}[(1)]
	\item[{\rm (1)}]\label{case g1} The Bergman metric of $M$ has constant negative holomorphic sectional curvature $\tau$;
	
	\item[{\rm (2)}]\label{case g2} $M$ is biholomorphic to a domain $\Omega \subset \mathbb{B}^n$, where $E := \mathbb{B}^n\setminus\Omega$ is a set of Lebesgue measure zero, and every $L^2$-holomorphic function on $\Omega$ extends to an $L^2$-holomorphic function on $\mathbb{B}^n$.
\end{enumerate}

Moreover, if either case holds, then $\tau = -{2(n + 1)^{-1}}$ and the Bergman kernel of $\Omega$ is the restriction of the Bergman kernel of $\mathbb{B}^n$ to $\Omega$.
\end{theorem}

In the following subsections,  we shall prove Theorem \ref{general version main theorem}. For that, we only need to show (1) implies (2), as the other implication is trivial. Assuming (1), by Theorem \ref{The portion of Huang and Li's Theorem we need}, $M$ is biholomorphic to a domain $\Omega \subset \mathbb{B}^n$ containing zero, where the Bergman kernel $K_{\Omega}$ satisfies (\ref{Huang and Li's formula of the Bergman kernel}). Throughout the rest of the proof, we will use the following terminology.
\begin{itemize}
	\item Throughout Section \ref{Section proof of theorem}, $\Omega, \phi, \lambda$ will refer to the objects described in Theorem \ref{The portion of Huang and Li's Theorem we need}.
	\item $\mu = (n + 1)\lambda$.
	\item $U$ will be a unitary linear map, $\Omega_U = U(\Omega) \subset \mathbb{B}^n$ and $\phi_U(z) = (\phi \circ U^{-1})(z)$ for $z \in \Omega_U$.  Additionally, we extend the definition of $\phi_U$ to $\mathbb{C}^n$ by setting it equal to 0 on the complement of $\Omega_U$.
	\item $\eta_U$ denotes the measure defined by $\eta_U = |\phi_U|^2\chi_{\Omega_U}m_{2n}$ and $\eta = \eta_{Id}$.
	\item $\mathbb{B}^n(z_0, r)$ will denote the ball in $\mathbb{C}^n$ with center $z_0$ of radius $r$ (and as before, $\mathbb{B}^n$ denotes the unit ball).
\end{itemize}

\subsection{Moment data from the Bergman kernel of $\Omega_U$}\label{subsection moment problem theorems}

In this section, we first construct an orthonormal basis for the domains $\Omega_U$.  Then we employ the moment theory to show that the measures $\{|\phi_U|^2\chi_{\Omega_U}m_{2n}\}$ are all the same and then make a crucial observation on the geometry of $\Omega$.

\begin{lem}\label{orthonormal basis of Bergman space}
For any unitary map $U$, the Bergman kernel of $\Omega_U$ is given by
    \begin{equation}\label{unitarily transformed bergman kernel written neater}
    K_{\Omega_U}(z, w) = \phi_{U}(z)\overline{\phi_U(w)}K_{\mathbb{B}^n}^{\lambda}(z, w), \quad z, w \in \Omega_U.
\end{equation}
Moreover, an orthonormal basis of the Bergman space $A^2(\Omega_U)$ is given by $\{c_{\alpha, \lambda}w^{\alpha}\phi_U(w)\}_{\alpha \in \mathbb{N}^n}$ where
    \begin{equation}\label{definition of calpha}
    c_{\alpha, \lambda} = \begin{cases}
        \sqrt{\left(n! \over \pi^n\right)^{\lambda}} & \alpha = \overrightarrow{0},
        \\
        \sqrt{\left(n! \over \pi^n\right)^{\lambda}{\mu \ldots (\mu + |\alpha| - 1) \over \alpha!}} & \alpha \neq \overrightarrow{0}.
    \end{cases}
    \end{equation}
    In particular,
    \begin{equation}\label{moment problem with phi}
{\delta_{\alpha\beta} \over c^2_{\alpha, \lambda}} = \int_{\mathbb{C}^n}w^{\alpha}\overline{w}^{\beta}\left( |\phi_U(w)|^2\chi_{\Omega_U}(w)dm_{2n}(w) \right), \quad \alpha, \beta \in \mathbb{N}^{n}.
\end{equation}

\end{lem}
\begin{proof}
Consider the map $f(z)=U^{-1}(z)$ from $\Omega_{U}$ to $\Omega$. By \eqref{Huang and Li's formula of the Bergman kernel} and the transformation law of the Bergman kernel under $f$, we get
\begin{equation}
K_{\Omega_U}(z, w)= K_{\Omega}(U^{-1}(z), U^{-1}(w)) = \phi_U(z)\overline{\phi_U(w)}K_{\mathbb{B}^n}^{\lambda}(U^{-1}(z), U^{-1}(w)), \quad z, w \in \Omega_U.
\end{equation}
By the unitary invariant property of $K_{\mathbb{B}^n}$, we obtain \eqref{unitarily transformed bergman kernel written neater}. Next, note that when $\beta \neq \overrightarrow{0}$
\begin{equation}
\left.{\partial^{|\beta|} \over \partial z^{\beta}}\right|_{z = 0} K_{\mathbb{B}^n}^{\lambda}(z, w) =   \left.{\partial^{|\beta|} \over \partial z^{\beta}}\right|_{z = 0}\left[ {n! \over \pi^n}(1 - \langle z, w \rangle)^{-(n + 1)}\right]^{\lambda}= \left(n! \over \pi^n\right)^{\lambda}\left[\mu\ldots (\mu + |\beta| - 1) \right] \overline{w}^{\beta}, \label{derivative of ball kernel}
\end{equation}
where we recall that $\mu = (n + 1)\lambda$. By the reproducing property of the Bergman kernel,
\begin{equation}\label{eqnzphi}
z^{\alpha}\phi_U(z) = \int_{\Omega_U} K_{\Omega_U}(z, w)w^{\alpha}\phi_U(w) dm_{2n}(w), \quad \alpha \in \mathbb{N}^{n}.
\end{equation}
By setting $z=w$ in \eqref{unitarily transformed bergman kernel written neater}, we see that $\phi_U$ never vanishes in $\Omega_U$. We bring \eqref{unitarily transformed bergman kernel written neater} into (\ref{eqnzphi}),  then cancel $\phi_U(z)$ and put in the formula of $K_{\mathbb{B}^n}(z, w)$, to get
\begin{equation}\label{z to the alpha equals}
z^{\alpha} = \int_{\Omega_U} w^{\alpha}|\phi_U(w)|^2 \left[ {n! \over \pi^n}(1 - \langle z, w \rangle)^{-n - 1}  \right]^{\lambda} dm_{2n}(w).
\end{equation}
Setting $\alpha = \overrightarrow{0}$ and  $z = 0$ gives
$$
1 = \left({n! \over \pi^n} \right)^{\lambda}\int_{\Omega_U}|\phi_U(w)|^2  dm_{2n}(w).
$$
By differentiating \eqref{z to the alpha equals},
\begin{eqnarray*}
    \alpha!\delta_{\alpha\beta}
    =& \int_{\Omega_U}w^{\alpha}|\phi_U(w)|^2\left.{\partial^{|\beta|} \over \partial z^{\beta}}\right|_{z = 0}\left[ {n! \over \pi^n}(1 - \langle z, w \rangle)^{-n - 1}  \right]^{\lambda} dm_{2n}(w)
    \\
    =& \left(n! \over \pi^n\right)^{\lambda}[\mu\ldots (\mu + |\beta| - 1)]\int_{\Omega_U}w^{\alpha}\overline{w}^{\beta}|\phi_U(w)|^2dm_{2n}(w).
\end{eqnarray*}
Consequently, \eqref{moment problem with phi} holds, and $\{c_{\alpha, \lambda}\phi_U(w)w^{\alpha}\}$ is an orthonormal set in $A^2(\Omega_U)$.
To deduce their completeness, we note that
\begin{eqnarray*}
K_{\Omega_U}(z, w) &=& \phi_U(z)\overline{\phi_U(w)}K_{\mathbb{B}^n}(z, w)^{\lambda}
\\
&=& \phi_U(z)\overline{\phi_U(w)}\sum_{\alpha} {1 \over \alpha!}z^{\alpha} \left.{\partial^{|\alpha|} \over \partial z^{\alpha}}\right|_{z = 0}\left[ {n! \over \pi^n}(1 - \langle z, w \rangle)^{-(n + 1)}\right]^{\lambda}
\\
&=& \sum_{\alpha} c_{\alpha, \lambda}\phi_U(z)z^{\alpha}c_{\alpha, \lambda}\overline{\phi_U(w)}\overline{w}^{\alpha},
\end{eqnarray*}
where in the second to last step we have used a Taylor series expansion in the $z-$variable at $z=0$. 
By the orthonormal series expansion of the Bergman kernel, $\{c_{\alpha, \lambda}\phi_U(z)z^{\alpha}\}$ must also be complete for the Bergman space of $\Omega_U$.
\end{proof}

\begin{remark}\label{recovering the Bergman space of the ball}
    When $\Omega = \mathbb{B}^n$, $\phi \equiv 1$, and $\lambda = 1$, Lemma \ref{orthonormal basis of Bergman space} recovers an orthonormal basis for the Bergman space of $\mathbb{B}^n$,  $\left\{\sqrt{\left(n! \over \pi^n \right){(|\alpha| + n)! \over n!\alpha!}}w^{\alpha}\right\}_{\alpha \in \mathbb{N}^n}
 $.
\end{remark}

Recall, for a unitary linear map $U,$ we write $\eta_U$ for the measure $|\phi_U|^2\chi_{\Omega_{U}}m_{2n}.$ 
Also recall that
$\eta:=\eta_{Id}=|\phi|^2\chi_{\Omega}m_{2n}$.

\begin{lem}\label{uniqueness of our complex moment problem}
For every unitary map $U$,
\begin{equation}\label{two measures solve the same moment problem}
\eta_U = \eta.
\end{equation}
Additionally, for each $\epsilon > 0$, there exists $z=(z_1, \ldots, z_n) \in \hbox{supp}(\eta_U)$ such that $|Re(z_1)| > 1 - \epsilon$.
\end{lem}

\begin{proof}
Since $\Omega_U \subset \mathbb{B}^n$, the measures $\eta_U$ have compact support. By the uniqueness of compactly supported solutions to the moment problem \eqref{moment problem with phi}, as described in \S\ref{moment problem subsection}, we conclude that \eqref{two measures solve the same moment problem} holds.
For the second part of the lemma, let us suppose, in order to reach a contradiction, that for some unitary transformation $U$, there exists $\epsilon > 0$ such that for all $z \in \hbox{supp}(\eta_U),$ $|Re(z_1)| \leq 1 - \epsilon$.  Then
\begin{equation}\label{Estimate we will show is impossible}
    \int_{\mathbb{C}^n}Re(z_1)^{2m} d\eta_U \leq (1 - \epsilon)^{2m}\eta_U(\mathbb{C}^n), \quad m \in \mathbb{N}.
\end{equation}
Notice that for $m > 1$, by (\ref{moment problem with phi}),
\begin{eqnarray}
    \int_{\mathbb{C}^n}Re(z_1)^{2m}d\eta_U = \int_{\mathbb{C}^n} \left({z_1 + \overline{z_1} \over 2}\right)^{2m} d\eta_U &=& \int_{\mathbb{C}^n} 2^{-2m}\sum_{k=0}^{2m} {2m \choose k} z_1^k\overline{z_1}^{2m - k}d\eta_U \nonumber
    \\
    &=& 2^{-2m}{2m \choose m}\int_{\mathbb{C}^n}|z_1^m|^2 d\eta_U \nonumber
    \\
    &=& \left(\frac{\pi^n}{n!}\right)^{\lambda}2^{-2m}\left({(2m)! \over m!\mu \ldots (\mu + m - 1)}\right). \label{value integral}
\end{eqnarray}
By Stirling's formula for the factorial and gamma function, there exist constants $A_1, A_2>0$ such that
$$A_1 \sqrt{k}\left(\frac{k}{e}\right)^k \leq k! \leq  A_2 \sqrt{k}\left(\frac{k}{e}\right)^k, \quad ~\text{for all integers}~~k \geq 1;$$
$$A_1 \frac{1}{\sqrt{x}}\left(\frac{x}{e}\right)^x \leq \Gamma(x) \leq  A_2 \frac{1}{\sqrt{x}} \left(\frac{x}{e}\right)^x, \quad ~\text{for all real numbers}~~x \geq 1.$$
Consequently, the quantity in (\ref{value integral}), modulo a non-zero constant factor that is independent of $m$, is greater than or equal to (denoted by $\gtrsim$)
\begin{eqnarray*}
2^{-2m}\left({\sqrt{2m}(2m)^{2m}e^{-2m} \over \sqrt{ m}(m^m)e^{-m}[\mu \cdots (\mu + m - 1)]} \right) &=& \left({m^m \over e^m[\mu\cdots (\mu + m - 1)]} \right)
\\
&=& \left({m^m \Gamma(\mu) \over e^m\Gamma(\mu + m)}\right)
\\
&\gtrsim& \left({m^m \sqrt{\mu + m}e^{\mu + m} \over e^m(\mu + m)^{\mu + m}} \right)
\\
&=& \left({m^m \over (\mu + m)^{\mu + m - {1 \over 2}}}\right)
\\
&=& \left({m \over \mu + m}\right)^m(\mu + m)^{{1 \over 2} - \mu}.
\end{eqnarray*}
By the above and (\ref{Estimate we will show is impossible}), (\ref{value integral}),  we obtain for $m > 1$
$$\left({m \over \mu + m}\right)^{m} \lesssim  (\mu + m)^{\mu- {1 \over 2}} (1 - \epsilon)^{2m}\eta_U(\mathbb{C}^n).$$
This is a contradiction, as the left hand side goes to $e^{-\mu}$, while the right hand side goes to $0$ when $m$ goes to infinity. This finishes the proof of the lemma.
\end{proof}

\subsection{Equality with a ball less a measure zero set}\label{subsection equality measure zero}

In this section, we prove that $\Omega$ is the unit ball less possibly a closed subset of measure zero. For that, two lemmas are needed.


\begin{lem}\label{indicator functions agree almost everywhere - unitary version}
 For any unitary transformation $U$, $\chi_{\Omega} = \chi_{\Omega_{U}}$ $m_{2n}$-almost everywhere.
\end{lem}
\begin{proof}

Let $E_{U} = \{z \in \mathbb{C}^n: \chi_{\Omega} \neq \chi_{\Omega_{U}}\}$. It suffices to prove $E_U$ is of Lebesgue measure zero. 
Suppose, in order to reach a contradiction,  that $m_{2n}(E_{U}) > 0$.  Then, either $E_{U} \cap \{\chi_{\Omega} = 0\}$ or $E_{U} \cap \{\chi_{\Omega} = 1\}$ has positive measure.  We only consider the former case, as the latter case is similar. Notice that
$
E_{U} \cap \{\chi_{\Omega} = 0\} = \{\chi_{\Omega} = 0, \chi_{\Omega_U} = 1\} \subset \Omega_U.
$
    Let $\{F_k\}$ be an exhaustion of $\Omega_U$ by precompact open subsets: $\cup F_k = \Omega_U$, $F_k \ssubset \Omega_U$ and $\overline{F_k} \subset F_{k + 1}$.  Then for sufficiently large $k$ ,
    $$
    m_{2n}(F_k \cap \{\chi_{\Omega} = 0, \chi_{\Omega_U} = 1\}) > 0.
    $$
As already observed in \S \ref{subsection moment problem theorems}, $\phi_U$ is nowhere zero and continuous in $\Omega_U$.
Hence,  $|\phi_U|^2$  has a positive lower bound on $F_k$. Consequently,
    $$
    \eta_U(F_k \cap \{\chi_{\Omega} = 0, \chi_{\Omega_U} = 1\}) = \int_{F_k \cap \{\chi_{\Omega} = 0, \chi_{\Omega_U} = 1\}} |\phi_U|^2 \chi_{\Omega_{U}}dm_{2n} > 0,
    $$
    while
    $$
    \eta(F_k \cap \{\chi_{\Omega} = 0, \chi_{\Omega_U} = 1\})= \int_{F_k \cap \{\chi_{\Omega} = 0, \chi_{\Omega_U} = 1\}} |\phi|^2 \chi_{\Omega}dm_{2n} = 0.
    $$
    This violates \eqref{two measures solve the same moment problem}.  Thus, $\chi_{\Omega} = \chi_{\Omega_{U}}$ $m_{2n}$-almost everywhere.
\end{proof}

\begin{lem}\label{After a unitary transformation this other set also has positive measure}
Let $z_0 \in \mathbb{B}^n$ and $\mathbb{B}^n(z_0, r) \subset \mathbb{B}^n$. If $m_{2n}((\mathbb{B}^n\setminus \Omega) \cap \mathbb{B}^n(z_0, r)) > 0$, then
$$
m_{2n}((\mathbb{B}^n\setminus \Omega) \cap \mathbb{B}^n(U(z_0), r)) > 0,
$$
for every unitary linear transformation $U$.
\end{lem}
\begin{proof}
We shall prove the contrapositive. Thus, suppose that $m_{2n}((\mathbb{B}^n\setminus\Omega) \cap \mathbb{B}^n(U(z_0), r)) = 0$ for some $U$.
Note that, by the invariance of the Euclidian metric and $m_{2n}$  under unitary transformations, we have
$\mathbb{B}^n(U(z_0), r)=U(\mathbb{B}^n(z_0, r))$ and
$m_{2n}(\mathbb{B}^n(z_0, r))=m_{2n}(\mathbb{B}^n(U(z_0), r))$. Since
$$
\mathbb{B}^n(U(z_0), r) = ((\mathbb{B}^n\setminus \Omega) \cap \mathbb{B}^n(U(z_0), r)) \cup (\Omega \cap \mathbb{B}^n(U(z_0), r)),
$$
we therefore obtain that
\begin{eqnarray*}
         m_{2n}(\mathbb{B}^n(z_0, r))=m_{2n}(\Omega \cap \mathbb{B}^n(U(z_0), r)) &=& \int_{\mathbb{B}^n(U(z_0), r)} \chi_{\Omega} dm_{2n}
        \\
        &=& \int_{\mathbb{B}^n(U(z_0), r)} \chi_{\Omega_U}dm_{2n}
        \\
        &=& \int_{\mathbb{B}^n(z_0, r)}\chi_{\Omega_U}(U(w)) dm_{2n}(w)
        \\
        &=& \int_{\mathbb{B}^n(z_0, r)}\chi_{\Omega}(w) dm_{2n}(w)=m_{2n}(\Omega \cap \mathbb{B}^n(z_0, r)),
    \end{eqnarray*}
where the third equality used Lemma \ref{indicator functions agree almost everywhere - unitary version} and the fifth equality used the fact that $\chi_{\Omega}(w) = \chi_{\Omega_U}(U(w))$ for $w \in \mathbb{C}^n.$ Consequently,
$m_{2n}((\mathbb{B}^n \setminus \Omega) \cap \mathbb{B}^n(z_0, r)) = 0.$
\end{proof}

\begin{pro}\label{Omega is a ball less a closed set of measure zero}
The domain $\Omega$ is the unit ball  $\mathbb{B}^n$ less possibly a relatively closed set of measure zero.
\end{pro}
\begin{proof}
    In order to reach a contradiction, we suppose that $m_{2n}(\mathbb{B}^n \setminus \Omega) > 0$. By the Lebesgue differentiation theorem applied to the characteristic function $\chi_{\bB^n\setminus\Omega}$ (cf. \cite[page 141, 7.12]{Ru87}), the density $D_{\mathbb{B}^n \setminus \Omega}(z)$ of $\mathbb{B}^n\setminus\Omega$ at $z$ satisfies that
    $$
    D_{\mathbb{B}^n\setminus \Omega}(z) := \lim_{r \to 0} {m_{2n}((\mathbb{B}^n\setminus\Omega) \cap \mathbb{B}^n(z, r)) \over m_{2n}(\mathbb{B}^n(z, r))} = 1,
    $$
for $m_{2n}$-almost all $z \in \mathbb{B}^n \setminus \Omega$. In particular,
$
\left\{z \in \mathbb{B}^n \setminus \Omega :\,  D_{\mathbb{B}^n\setminus \Omega}(z) = 1\right\} \neq \emptyset;
$
let $z_0$ be a member of this set. Then, for all $r > 0$ sufficiently small,
    $
    m_{2n}((\mathbb{B}^n\setminus\Omega) \cap \mathbb{B}^n(z_0, r)) > 0.
    $
Since $\Omega$ is open, $z_0$ cannot be in $\Omega$. Now, let $U$ be a unitary transformation.  By Lemma \ref{After a unitary transformation this other set also has positive measure}, for all $r > 0$ sufficiently small,
    $$
    m_{2n}((\mathbb{B}^n\setminus\Omega) \cap \mathbb{B}^n(U(z_0), r)) > 0.
    $$
Again, since $\Omega$ is open,  it follows that $U(z_0) \not\in \Omega$.  As $U$ was an arbitrary unitary transformation, setting $S_{|z_0|}:=\{z : |z| = |z_0|\},$ we conclude that
    $$
    S_{|z_0|} \cap \Omega = \emptyset.
    $$
    Since $z_0 \in \mathbb{B}^n$ and $\Omega$ contains 0, we have $|z_0|>0$. Furthermore, since by Lemma \ref{uniqueness of our complex moment problem}, $\Omega$ has points with modulus arbitrarily close to 1, the sphere $S_{|z_0|}$ disconnects $\Omega$.  This is impossible, as $\Omega$ is connected.  Therefore, we must have $m_{2n}(\mathbb{B}^n\setminus\Omega) = 0$.
\end{proof}

\subsection{Finishing the proof}\label{subsection finishing proof}


To finish the proof of Theorem \ref{general version main theorem} (which implies Theorem \ref{main theorem - negative constant case}), we first establish

\begin{pro}\label{value phi lambda}
	The function $\phi \equiv 1$ and $\lambda=1$. Consequently,
	$K_{\Omega}= K_{\mathbb{B}^n}$ on  $\Omega$. Moreover,  every $f \in A^2(\Omega)$ extends holomorphically to a function in  $A^2(\mathbb{B}^n)$.
\end{pro}

\begin{proof}
	By \eqref{two measures solve the same moment problem} and Lemma \ref{indicator functions agree almost everywhere - unitary version}, $|\phi_U|^2\chi_{\Omega}m_{2n} = |\phi|^2\chi_{\Omega}m_{2n}$ for every unitary map $U$. This implies $|\phi_U|^2= |\phi|^2$ in $\Omega$ $m_{2n}-$almost everywhere. By continuity, $|\phi_U|^2= |\phi|^2$ holds everywhere in $\Omega \cap \Omega_U$. Then $\phi/(\phi_U)$ is a constant function in $\Omega \cap \Omega_U$ by the open mapping theorem. By inspecting the value at $0$ in $\Omega \cap \Omega_U$, we conclude that the constant is one; i.e., $\phi =\phi_ U$ in $\Omega \cap \Omega_U$ for every unitary map $U$. This means that $\phi$ is constant on every small sphere near $0 \in \Omega.$ Hence, $\phi \equiv b$ in $\Omega$ for some constant $b>0$ (recall that $\phi(0)>0$).	
	
By Lemma \ref{orthonormal basis of Bergman space} and Remark \ref{recovering the Bergman space of the ball}, $\{bc_{\alpha, \lambda}w^{\alpha}\}_{\alpha \in \mathbb{N}^n}$ and $\left\{\psi_{\alpha}=\sqrt{\left(n! \over \pi^n \right){(|\alpha| + n)! \over n!\alpha!}}w^{\alpha}\right\}_{\alpha \in \mathbb{N}^n}$  give  orthonormal bases of $A^2(\Omega)$ and $A^2(\mathbb{B}^n)$, respectively. But $\Omega$ and $\mathbb{B}^n$ just differ by a measure zero set and, thus, the functions $w^{\alpha}$ have the same $L^2-$norm on the two domains. Consequently, for all $\alpha \in \mathbb{N}^n,$
$$bc_{\alpha, \lambda}=\sqrt{\left(n! \over \pi^n \right){(|\alpha| + n)! \over n!\alpha!}}.$$
	By inspecting the two equations obtained by setting $\alpha = \overrightarrow{0}$ and $|\alpha|=1$, one easily sees that $\mu=n+1$ and $b=1$; i.e., $\lambda=1$ and $\phi \equiv 1.$ Then, by (\ref{Huang and Li's formula of the Bergman kernel}), $K_{\Omega}= K_{\mathbb{B}^n}$ on  $\Omega$.
	Furthermore, by the above discussion, $A^2(\Omega_U)$ and $A^2(\mathbb{B}^n)$ have the same  orthonormal basis $\{\psi_{\alpha}\}_{\alpha \in \mathbb{N}^n}.$
	If $f \in A^2(\Omega)$, then $f(w) = \sum_{\alpha \in \mathbb{N}^n} a_{\alpha}\psi_{\alpha}(w)$ for a sequence $\{a_{\alpha}\}_{\alpha \in \mathbb{N}^n}$ with $\sum_{\alpha \in \mathbb{N}^n} |a_{\alpha}|^2 < \infty$.  Since $\sum_{\alpha \in \mathbb{N}^n} a_{\alpha}\psi_{\alpha}(w)$ also defines a function in the Bergman space of $\mathbb{B}^n$, $f$ extends holomorphically to a function in $A^2(\mathbb{B}^n)$.
\end{proof}

We are now ready to prove Theorem \ref{general version main theorem}.

\begin{proof}[{\bf Proof of Theorem \ref{general version main theorem}.} ]

It is trivial that (2) implies (1) as it follows immediately from the biholomorphic invariant property of the Bergman metric. We only need to prove that
(1) implies (2).

Now suppose that
(1) holds; that is, the Bergman metric has its holomorphic sectional curvature equal to a negative constant. By Theorem \ref{The portion of Huang and Li's Theorem we need}, $M$ is biholomorphic to a domain $\Omega \subset \mathbb{B}^n$, whose Bergman kernel is given by $K_{\Omega}(z, w) = \phi(z)\overline{\phi(w)}K_{\mathbb{B}^n}^{\lambda}(z, w)$.  Then by Proposition \ref{Omega is a ball less a closed set of measure zero} and \ref{value phi lambda}, the statements in
(2) hold.
Finally, by the biholomorphic invariant property of the Bergman metric again, the holomorphic sectional curvature of $M$, is the same as that of $\mathbb{B}^n$ and $\Omega$, which is $-2(n + 1)^{-1}$.
\end{proof}

\begin{proof}[{\bf Proof of Theorem \ref{main theorem - negative constant case}}] The Bergman space of any bounded domain in $\mathbb{C}^n$ separates points.  Therefore, Theorem \ref{main theorem - negative constant case} is a special case of Theorem \ref{general version main theorem}.
\end{proof}

We conclude this section with a proof of Proposition \ref{pronegligibleset}.

\medskip

{\bf Proof of Proposition \ref{pronegligibleset}.} The ``if'' implication is trivial, so we only need to prove the ``only if'' implication. For that, we assume the Bergman kernels $K_{\hat{U}}$ and $K_U$ are equal on $\hat{U}$. Write $F=U \setminus \hat{U}.$ We first claim that $F$ is of Lebesgue measure zero. In order to reach a contradiction, suppose that $m_{2n}(F) > 0$. As in the proof of Proposition \ref{Omega is a ball less a closed set of measure zero}, by the Lebesgue differentiation theorem applied to the characteristic function of $F$, we see there exists $z_0 \in F$ such that for all $r > 0$ sufficiently small,
$m_{2n}(F \cap \mathbb{B}^n(z_0, r)) > 0.$ We also note that the set $\{z \in U : K_U(z, z_0)=0\}$, if nonempty, is a complex hypervariety in $U$. Consequently, since $K_U(z,z_0)=\overline{K_U(z_0,z)}$, there exists $p \in \hat{U}$ such that $K_U(z_0, p) \neq 0.$  This implies
$$\int_F |K_U(z, p)|^2 dm_{2n}(z) >0.$$
On the other hand, the reproducing property of the Bergman kernels implies

$$\int_U |K_U(z, p)|^2 dm_{2n}(z)=K_U(p,p)=K_{\hat{U}}(p,p)= \int_{\hat{U}} |K_{\hat{U}}(z,p)|^2 dm_{2n}(z).$$
By the assumption on the Bergman kernels and using a complexification argument, we have $K_{\hat{U}}(\cdot,p)=K_U(\cdot, p)$ on $\hat{U}.$ Then the above equation yields $\int_F |K_U(z, p)|^2 dm_{2n}(z)=0,$ which is a contradiction. Hence, we must have  $m_{2n}(F) = 0$. Consequently, an orthonormal basis $\{\phi_i\}_{i=1}^{\infty}$ of $A^2(U)$ is an orthonormal set in $A^2(\hat{U}).$ But $K_{\hat{U}}=K_U$ on $\hat{U}$, which implies that $\{\phi_i\}_{i=1}^{\infty}$ is also an orthonormal basis of $A^2(\hat{U}).$ Finally we conclude that every function in $A^2(\hat{U})$ extends holomorphically to a function in $A^2(U)$ by repeating the last part in the proof of Proposition \ref{value phi lambda}. This proves $F$ is a Bergman-negligible subset of $U$. \qed

\section{Proof of Corollary \ref{cor2} and \ref{cor3}}\label{sectioncor23}

We prove Corollary \ref{cor2} and \ref{cor3} in this section.  As a preparation, we start with following topological fact about domains with $C^0$-boundary.

\begin{pro}\label{simply-connected region from C0 boundary}
	Let $D$ be a domain in $\mathbb{C}^n$.  Assume that $D$ has $C^0$-boundary at a point $q \in \partial D$.   Then there exists a neighborhood $O$ of $q$ in $\mathbb{C}^n$ such that $D \cap O$ is (connected and) simply connected.
\end{pro}
\begin{proof}
By Definition \ref{defnc0bdry}, there exists a homeomorphism $H$ from a neighborhood $O$ of $q$ in $\mathbb{C}^n$ to a domain $O_1 \times O_2 \subset \mathbb{R}^{2n-1} \times \mathbb{R}$. Moreover, $H$ maps
$D \cap O$ onto
$W:=\{(x, y) \in O_1 \times O_2: y < \phi(x)\}$
for some $C^0$-function $\phi$ on $O_1$.
By shrinking $O$ and shifting, we may assume that $O_1=\{x \in \mathbb{R}^{2n -1}: |x| < \delta_1\}$ and $O_2=\{y \in \mathbb{R}: |y| < \delta_2\}$ for some $\delta_1, \delta_2 > 0$, and $\phi(0) = 0, H(q)=(0,0)$.  By shrinking $\delta_1$, we may assume that
\begin{equation}\label{eqnphi}
|\phi(x)| < {\delta_2 \over 2}, ~\text{for all}~x \in O_1.
\end{equation}
Fix  $T \in (-\delta_2, -{\delta_2 \over 2})$, and fix any point $p=(x_0, y_0) \in W$. Then it is clear that $p_1:=(x_0, T)$ and $p_0:=(0, T)$ are both in $W$. Now, $p$ is connected to $p_1$ by a path in $W$: $u(t)=(x_0, (1-t)y_0+tT)$, $0 \leq t \leq 1.$ Moreover, $p_1$ is connected to $p_0$ by the path
$v(t)=((1-t)x_0, T)$, $0 \leq t \leq 1,$ in $W$. Therefore $W$ is path-connected, as all points in $W$ can be connected to $p_0.$ To establish the simply connectedness, it then suffices to prove:

\begin{claim}
	If $\gamma(t) = (x(t), y(t))$, $0 \leq t \leq 1$, is a loop in $W$ based at $p_0$, then $\gamma$ is null-homotopic in $W$.
\end{claim}
\begin{proof}
	\renewcommand{\qedsymbol}{$\blacksquare$}
	Define a loop homotopy with fixed base point $p_0$ in $W$:
	$$
	H_1(s, t) = (x(t), \mu(s,t)) := (x(t), (1 - s)y(t) + s T),
	$$
where $(s, t) \in [0, 1]^2$.  Since $-\delta_2 < \mu(s,t) < \phi(x(t))$, it follows that $H_1(s, t) \in W$ for all $s,t$.  Via $H_1$, $\gamma$ is deformed to a new loop $\alpha \subset W$ defined by $\alpha(t) = (x(t), T), 0 \leq t \leq 1$.
Define another loop homotopy with fixed base point $p_0$ in $W$: $H_2(s,t)=((1-s)x(t), T),$ where $(s, t) \in [0, 1]^2$. It is clear that $\alpha$ gets deformed to the point $p_0,$ via the homotopy $H_2.$ 
Therefore, $\alpha,$ and thus $\gamma,$ are null-homotopic in $W$.
\end{proof}
This establishes Proposition \ref{simply-connected region from C0 boundary}.
\end{proof}

Using Proposition \ref{simply-connected region from C0 boundary}, we can further obtain the following result.

\begin{pro}\label{Proposition variety cannot map to C0 boundary}
Let $D$ be a bounded domain in $\mathbb{C}^n$, and $M \subset \partial D$ be the set of all $C^0$-boundary points of $D$.  Let $V$ be a (nontrivial) complex hypervariety of $\mathbb{B}^n$.  Suppose that $G$ is a biholomorphism 
from $\mathbb{B}^n \setminus V$ onto $D$. Then, $G$ extends as a holomorphic map from $\mathbb{B}^n$ to $\overline{D}$, still denoted by $G$, such that $G(V)\subset\partial D$ and $M \cap G(V) = \emptyset$. 
\end{pro}

\begin{proof} The fact that $G$ extends as a holomorphic map from $\mathbb{B}^n$ to $\overline{D}$ follows immediately from the removable singularity theorem, since $D$ is bounded. The fact that $G(V)\subset\partial D$ follows by observing that $V\subset \partial(\bB^n\setminus V)$ and $G$ is, in particular, a proper map $\mathbb{B}^n \setminus V\to D$.
Now, suppose that $G(V) \cap M \neq \emptyset$.  Then there exists $p \in V$ such that $G(p) \in M$. By the definition of $C^0$-boundary, $M$ is an open subset of $\partial D.$ Consequently, as $G(V) \subset \partial D,$ $G$ maps an open neighborhood of $p$ in $V$ to $M$. Therefore, by perturbing $p$ in $V$, we may assume that $p$ is a smooth point of some irreducible component $V_0$ of $V$ and that $q=G(p) \in M.$ By Proposition \ref{simply-connected region from C0 boundary}, there exists a neighborhood $O$ of $q$ in $\mathbb{C}^n$ such that $D \cap O$ is simply connected. Since $G$ is continuous at $p$,
there is some neighborhood $U$ of $p$ in $\mathbb{B}^n$ such that $G(U) \subset O.$ 
Let $h$ be a holomorphic function in $\bB^n$ such that $V=h^{-1}(0)$. Since $p$ is a smooth point of $V_0$ 
there is a small loop $\beta \subset U \setminus V$ near $p$ such that any branch of $\log h$ has nontrivial monodromy along $\beta$. This implies $\beta$ is not null-homotopic in $\mathbb{B}^n \setminus V$.
	On the other hand, since $\beta \subset U \setminus V,$ we see $\alpha:=G(\beta)$ is contained in  $D \cap O$, which is simply-connected.
As a result, $\alpha$ is null-homotopic in $D \cap O$, and thus in $D$. This is a contradiction, as $G$ is a biholomorphism from $\mathbb{B}^n \setminus V$ to $D$. Thus, $M \cap G(V) = \emptyset$, as desired.
\end{proof}

We also include the following easy fact.

\begin{pro}\label{pro needed for the proof of Corolllary 2 and Corollary 3}
	Let $D \subset \mathbb{C}^n$ be a bounded domain, and $M \subset \partial D$ the set of all $C^0$-boundary points of $D$.  Let $\Omega \subset \mathbb{B}^n$  be a domain such that $E:=\mathbb{B}^n \setminus \Omega$ is contained in $\partial \Omega$.  Let $G:\mathbb{B}^n \to \mathbb{C}^n$ be a holomorphic map such that $G$ restricts to a biholomorphism $G|_{\Omega}$ from $\Omega$ onto $D$, and denote by $V_G$ the locus of points in $\mathbb{B}^n$ where the Jacobian of $G$ has determinant equal to 0. Then $M \cap G(E \setminus V_G) = \emptyset$.
\end{pro}
\begin{proof}
	Suppose, in order to reach a contradiction, that there is a point $p \in E \setminus V_G$ such that $q=G(p) \in M$.  Since $G$ is locally biholomorphic at $p$, there exists a small neighborhood $U \subset \mathbb{B}^n$ of $p$ such that $O:= G(U)$ is an open neighborhood of $q$ in $\mathbb{C}^n$. But $G$ maps $E \cap U$ to $\partial D$ (since $E \subset \partial \Omega$ as in the proof of Proposition \ref{Proposition variety cannot map to C0 boundary}) and $U \setminus E$ to $D$.  That is, $O$ is an open neighborhood of $q$ containing only interior points and boundary points of $D$.  This is impossible, as $q$ is a $C^0$-boundary point of $D$.
\end{proof}

We are now ready to prove the first two corollaries and will start with Corollary \ref{cor2}.

\medskip

{\bf Proof of Corollary \ref{cor2}.} By the hypothesis and Theorem \ref{main theorem - negative constant case}, 
$D$ is biholomorphic to $\Omega = \mathbb{B}^n\setminus E$, where $E$ is a Bergman-negligible subset of $\mathbb{B}^n$.  Let $G$ be a biholomorphic map from $\Omega$ to $D$.  Since $D$ is bounded, $G$ extends to a holomorphic map $\mathbb{B}^n\to\overline{D}$.  As above, we denote the extended map by $G$ as well. Let $V_G$ denote the (possibly empty) locus of points in $\mathbb{B}^n$ where the determinant of the Jacobian of $G$ equals 0.  Since $G$ is biholomorphic in $\Omega$, we have $V_G \subset E$. Note also that $G(E) \subset \partial D,$ and  $D$ has $C^0$-boundary at every point of $\partial D$. Hence, by Proposition \ref{pro needed for the proof of Corolllary 2 and Corollary 3}, $E \setminus V_G= \emptyset,$ and thus $V_G = E$. By Proposition \ref{Proposition variety cannot map to C0 boundary}, $E= \emptyset$. This establishes Corollary \ref{cor2}. \qed

\medskip

We now prove Corollary \ref{cor3}. 

\medskip

{\bf Proof of Corollary \ref{cor3}.} Since $q \in \partial D$ is assumed to be a boundary orbit accumulation point, by definition, there is a sequence $\{\phi_j\}_{j=1}^{\infty} \subset \mathrm{Aut}(D)$ and $p \in \Omega$ such that $\phi_j(p) \to q$ as $j \to \infty$.
 We first establish the following lemma.

\begin{lem}\label{Lemma for phij converge to phi locally uniformly}
The maps $\phi_j$ converges locally uniformly on $D$ to the constant function $\phi \equiv q$.
\end{lem}

\begin{proof}
Since $D$ is bounded, by Montel's Theorem, every subsequence of $\{\phi_j\}_{j \geq 1}$ has a further subsequence that converges locally uniformly on $D$. It therefore suffices to prove that all convergent subsequences of $\{\phi_j\}_{j \geq 1}$ have the same limit $\phi \equiv q$. For this, fix any such convergent subsequence, which we denote again by $\{\phi_j\}_{j \geq 1}$ for simplicity, and let $\phi$ denote its limit.
Then, either $\phi \in \hbox{Aut}(D)$ or $\phi(D) \subset \partial D$ (see, e.g., \cite[Theorem 5.1.2]{Kr13}).  By hypothesis, $\phi(p) = \lim_{j \to \infty}\phi_j(p) = q \in \partial D$.  Hence, the second alternative holds, $\phi(D) \subset \partial D$.  
Let now $W$ be a small open neighborhood of $p$ such that $\phi(z) \in O \cap \partial D$ for every $z \in W$, where $O$ is the neighborhood of $q$ given in the assumption of Corollary \ref{cor3}.
Since $O \cap \partial D$ contains no (nontrivial) complex varieties, $\phi(W) = \{q\}$.  By the analyticity of $\phi$, it follows that $\phi \equiv q$, as desired.
\end{proof}

By Lemma \ref{Lemma for phij converge to phi locally uniformly},
$\lim_{j \to \infty} \phi_j(z) = q$ for every $z \in D.$ Fix an arbitrary $z \in D$ and let $z_j = \phi_j(z)$. By assumption, we have
$$
\lim_{j \to \infty} \sup_{X \in T_{z_j}D\setminus\{0\}} \left| H(g_D)(z_j, X) - \tau \right| = 0.
$$
Since $\phi_j \in \hbox{Aut}(D)$ and the Bergman metric is invariant under $\hbox{Aut}(D)$, we have $H(g_D)(z_j, d\phi_j(X))=H(g_D)(z, X)$, and we conclude that
\begin{equation}\label{constant holo sectional curvature in the neighborhood W}
	H(g_D)(z, X) \equiv \tau, \quad z \in D, \quad X \in T_z(D)\setminus\{0\}.
\end{equation}
That is, $g_D$ has constant holomorphic sectional curvature $\tau$ in $D$. Then, by Theorem \ref{main theorem - negative constant case}, $\tau = -{2(n + 1)^{-1}}$ and $D$ is biholomorphic to $\mathbb{B}^n$ less a Bergman-negligible subset $E$.  This proves the first assertion of Corollary \ref{cor3}.

To prove the second assertion, we assume that $D$ has $C^0$-boundary at $q$. We will prove $E=\emptyset,$ which then implies that $D$ is biholomorphic to $\mathbb{B}^n$, completing the proof. For this, let $G$ be a biholomorphic map from $\Omega$ to $D$, and $F: D \to \Omega$ its inverse. Again, $G$ extends to a holomorphic map on $\mathbb{B}^n$, which we continue to denote by $G$. As above (in the proofs of Propositions \ref{Proposition variety cannot map to C0 boundary} and \ref{pro needed for the proof of Corolllary 2 and Corollary 3}), we have $G(E) \subset \partial D.$ 
As in the proof of Corollary \ref{cor2},  let $V_G$ denote the locus of points in $\mathbb{B}^n$ where the determinant of the Jacobian of $G$ equals 0.
We again have $V_G \subset E$.
Let $E^{-} := E \setminus V_G$, and set $\Omega^{+} := \mathbb{B}^n\setminus V_G = \Omega \cup E^{-}$, which is clearly a domain in $\mathbb{C}^n$.  Let $D^+ = G(\Omega^+)$.  One readily observes that $G$ is one-to-one and locally biholomorphic  on $\Omega^+$. Consequently, $G$ is biholomorphic on $\Omega^+$ and $D^+$ is a domain in $\mathbb{C}^n$.  The inverse of $G|_{\Omega^+}$ is the holomorphic extension of $F$, initially defined on $D$, to $D^+$.  We will denote the extension also by $F$. By the definition of  $C^0$-boundary, there exists a small neighborhood $O$ of $q$ in $\mathbb{C}^n$, such that $M_0:=O \cap \partial D$ consists of only $C^0$-boundary points of $D$. By Proposition \ref{pro needed for the proof of Corolllary 2 and Corollary 3}, $G(E^{-}) \subset \partial D \setminus M_0$. Since $D^+ = D \cup G(E^-)$, it follows that 
$$
O \cap D^+ = O \cap D.
$$
As a result, $M_0$ is also an open $C^0$-boundary piece of  $D^+$.
If $V_G$ is nonempty, we apply Proposition \ref{Proposition variety cannot map to C0 boundary} to the biholomorphism $G|_{\Omega^+}$ from $\Omega^{+}$ to $D^+$, and obtain $M_0 \cap G(V_G)=\emptyset$ (If $V_G$ is empty, the conclusion is trivial). Consequently, as $E=E^- \cup V_G,$ we have
\begin{equation}\label{eqnge}
G(E) \subset \partial D \setminus M_0.
\end{equation}
Now, let  $\psi_j := F \circ \phi_j \circ G|_{\Omega}$, for $j \geq 1$. It is clear that $\psi_j \in \hbox{Aut}(\Omega)$, and thus it preserves the Bergman metric of $\Omega$, which equals the Bergman metric of $\mathbb{B}^n.$ Consequently, by Calabi \cite{Ca53} (see also \cite[Theorem 1.1]{HL12}), $\psi_j$  extends as an automorphism of $\mathbb{B}^n$, which we still denote by
$\psi_j \in \hbox{Aut}(\mathbb{B}^n).$ Since $E=\mathbb{B}^n \setminus \Omega,$ we must have $\psi_j(E) = E$.
Next set
$$
\phi_j^+ = G \circ \psi_j \circ F|_{D^+}.
$$
It is clear that $\phi_j^{+}$ is a holomorphic map on $D^+$ and $\phi_j^+|_D = \phi_j \in \hbox{Aut}(D)$.  Thus, $\phi_j^+(D^+) \subset \overline{D}$ and, hence, $\phi_j^+$ is bounded on $D^+$.  By Montel's theorem, after passing to a subsequence if needed, we can assume that $\phi_j^+$ converges to $\phi^+$ locally uniformly in $D^+$ for some holomorphic map $\phi^+$ in $D^+$.  However, since $\phi_j^+|_D = \phi_j \to \phi$, where $\phi \equiv q$, it follows that $\phi^+ \equiv q$.  We next prove:

\begin{claim} The set $E^-$ is empty.
\end{claim}
\begin{proof}
	\renewcommand{\qedsymbol}{$\blacksquare$}
Suppose that $E^- \neq \emptyset$ and let $w^* \in E^- = E \setminus V_G$.  Let $z^* = G(w^*) \in \partial D \cap D^+$.  Then $\phi_j^+(z^*)$ approaches $q$ as $j$ tends to $\infty$.  Notice that
$$
\phi_j^+(z^*) = G \circ \psi_j \circ F(z^*) = G \circ \psi_j(w^*).
$$
Write $w_j := \psi_j(w^*)$ and recall that $\psi_j(E) = E$, which implies that $w_j \in E$.  Thus, $G(w_j) \in \partial D$.  Since $G(w_j) = \phi_j^+(z^*)$ tends to $q$ as $j\to \infty$, for sufficiently large $j$ we then have that $G(w_j) \in M_0$.  This contradicts \eqref{eqnge}.
\end{proof}


Consequently, $E = V_G$.  We will now show that $V_G = \emptyset$, which will complete the proof of Corollary \ref{cor3}.  Suppose that $V_G\neq\emptyset$. We pick a smooth point $\hat{w}$ on some irreducible component of $E = V_G$.  Then there exists a small analytic disk $A$ such that $\overline{A} \subset \mathbb{B}^n$ and $\overline{A} \cap E = \{\hat{w}\}$.  Thus, $\overline{A} \setminus \{\hat{w}\} \subset \Omega$.  As before, let
$$
\psi_j = F \circ \phi_j \circ G \in \hbox{Aut}(\Omega) \cap \hbox{Aut}(\mathbb{B}^n),
$$
and let $K := G(\partial A) \ssubset D$.  By Lemma \ref{Lemma for phij converge to phi locally uniformly}, $\phi_j$ approaches $q$ uniformly on $K$ as $j\to\infty$.  Note that $\phi_j(K) = G \circ \psi_j(\partial A)$ and thus
$$
\max_{w \in \partial A} \|G \circ \psi_j(w) - q\| \to 0, \quad j \to \infty.
$$
Since $G \circ \psi_j$ is holomorphic in $A$ and continuous on $\overline{A}$, the maximum principle implies that
$$
\|G \circ \psi_j(\hat{w}) - q\| \to 0, \quad j \to \infty.
$$
Recall that $\psi_j(E) = E$, which implies that $\hat{w}_j := \psi_j(\hat{w}) \in E = V_G$.  Since $G(\hat{w}_j) \to q$ as $j \to \infty$, for sufficiently large $j$ we have $G(\hat{w}_j) \in M_0$.  This contradicts (\ref{eqnge}), demonstrating that $E=\emptyset,$ as desired.
\qed

\bigskip

\section{Examples}
\label{sectionexamples}

We give a few examples to justify the statements in Remark \ref{rmkcor23}. The first example will be used in the construction given in Example \ref {example 2 for corollary 23} below.

\begin{example}\label{example 1 for corollary 23}
	Let $D_1 = \mathbb{B}^2 \setminus \{(z_1, z_2) \in \mathbb{B}^2:\, z_2  = 0\}$.  Let
	$$
	\phi_j(z_1, z_2) = \left( {a_j-z_1 \over 1 - \overline{a_j}z_1}, {(1 - |a_j|^2)^{1 \over 2} \over 1 - \overline{a_j}z_1}z_2 \right), \quad j \geq 1,
	$$
where $\{a_j\}_{j=1}^{\infty} \subset \mathbb{C}$ is a sequence of points with $|a_j| <1$ for all $j$ and $a_j \to 1$ as $j \to \infty$. Let $\phi$ be the constant map
$\phi(z_1, z_2) \equiv (1, 0)$. It is easy to verify that every $\phi_j \in \hbox{Aut}(\mathbb{B}^2) \cap \hbox{Aut}(D_1)$, and $\phi_j$ converges to $\phi$ locally uniformly on $\mathbb{B}^2$, as $j \to \infty.$ Consequently, $(1, 0)$ is a boundary orbit accumulation point of $D_1$.
Moreover, since the Bergman kernel of $D_1$ is the restriction of that of $\mathbb{B}^2$,
the Bergman metric of $D_1$ has constant holomorphic sectional curvature $-\frac{2}{3}.$
\end{example}

\medskip

\begin{example}\label{example 2 for corollary 23}
Let $D_1$ be as in Example \ref{example 1 for corollary 23}, and let $D_2$ be the bounded domain in $\mathbb{C}^2$:
$$\{z = (z_1, z_2) \in \mathbb{C}^2:  |z_1|^2 + |z_2|^2(|z_2|^2 - 1) < 0\}.$$
Note that $D_2$ is the biholomorphic image of $D_1$ under the map $\Phi(z) = (z_1z_2, z_2)$.  Note also that $\Phi$ maps the variety  $V:=\{(z_1, z_2) \in \mathbb{C}^2:\, z_2 = 0\}$ to $(0, 0) \in \partial D_2,$ and $\partial D_2$ is the image of $\partial \mathbb{B}^2$ under $\Phi$.  Thus, $\partial D_2$ contains no (nontrivial) complex varieties.  Let $\varphi_j = \Phi \circ \phi_j \circ \Phi^{-1}$ where $\phi_j$ is the automorphism of $D_1$ from Example \ref{example 1 for corollary 23}.  Clearly, for $j \geq 1$, $\varphi_j$ is an automorphism of $D_2$.  Moreover, as $j$ tends to infinity, $\varphi_j$ converges to the constant map $\varphi \equiv (0, 0)$ locally uniformly on $D_2$. Therefore, $q:=(0,0)=\Phi(1,0)$ is a boundary orbit accumulation point of $D_2.$ Since $D_2$ is biholomorphic to $D_1$, the Bergman metric of $D_2$ has constant holomorphic sectional curvature equal to $-\frac{2}{3}.$
On the other hand, since $\Phi(V)=q$, by Proposition \ref{Proposition variety cannot map to C0 boundary}, $D_2$ cannot have $C^0$-boundary at $q$. It is clear that $D_2$ has $C^0$-boundary at all other points of $\partial D_2.$ Furthermore, $D_1$, and therefore also $D_2,$ is not biholomorphic to $\mathbb{B}^2.$ Thus, by the above discussion, $D_2$ satisfies all the assumptions of Corollary \ref{cor3} except the $C^0$-smoothness at $q.$ This example justifies the assertions in Remark \ref{rmkcor23}.
\end{example}

The above examples in $\mathbb{C}^2$ can be easily generalized to higher dimensions as follows.

\begin{example}\label{example 3}
Let $W_1=\mathbb{B}^n \setminus \{z=(z',z_n)=(z_1, \ldots,z_{n-1}, z_n) \in \mathbb{B}^n: z_n=0\}.$ Let $\{A_j(z')\}_{j=1}^{\infty}$ be a sequence of
automorphisms of $\mathbb{B}^{n-1} \subset \mathbb{C}_{z'}^{n-1}$ such that $A_j \to A$ for some constant map $A \equiv p_0$ on every compact subset of $\mathbb{B}^{n-1}.$ Here $p_0$ is a point on $\partial \mathbb{B}^{n-1}.$ Let $T_j(z')$ denote the determinant of the Jacobian of $A_j(z')$. Note that $T_j$ is nowhere zero in $\mathbb{B}^{n-1}$, but $T_j \to 0$ on every compact subset of $\mathbb{B}^{n-1}.$  Let
$$\phi_j(z)=\left(A_j(z'), \left(T_j(z')\right)^{\frac{1}{n}}z_n \right).$$
Here we pick some branch of $\left(T_j(z')\right)^{\frac{1}{n}}$ in $\mathbb{B}^{n-1}$. By the transformation formula of the Bergman kernel of $\mathbb{B}^{n-1},$ we see $(1-|z'|^2)^{n} |T_j(z^\prime)|^2=(1-|A_j(z')|^2)^n .$  Using this fact, it is easy to verify that every $\phi_j$ belongs to $\hbox{Aut}(\mathbb{B}^n) \cap \hbox{Aut}(W_1)$. We also note that $\phi_j$ converges to the constant map $\phi \equiv (p_0,0),$ locally uniformly on $\mathbb{B}^2$, as $j \to \infty.$ Let $W_2$ be the bounded domain in $\mathbb{C}^n:$
$$\{z = (z', z_n) \in \mathbb{C}^n:  |z'|^2 + |z_n|^2(|z_n|^2 - 1) < 0\}.$$
Note that $W_2$ is the biholomorphic image of $W_1$ under the map $\Phi(z) = (z_1z_n, \ldots, z_{n-1}z_n, z_n)$.
Moreover, $\Phi$ maps the variety  $\{ z \in \mathbb{C}^n:\, z_n = 0\}$ to $(0, 0) \in \partial W_2,$
Let $\varphi_j = \Phi \circ \phi_j \circ \Phi^{-1}$. Then, each $\varphi_j$ is an automorphism of $W_2$.  Moreover, $\varphi_j$ converges to the constant map $\varphi \equiv (0, 0)$ locally uniformly on $W_2$. Consequently, $q:=(0,0)=\Phi(p_0,0)$ is a boundary orbit accumulation point of $W_2.$
As in Example \ref{example 2 for corollary 23},  the Bergman metrics of $W_1$, and thus $W_2$, have constant holomorphic sectional curvature equal to $-2(n + 1)^{-1}$.
Moreover, since $\partial W_2$ is the image of $\partial \mathbb{B}^n$ under $\Phi$, $\partial W_2$ has no (nontrivial) complex varieties. The domain $W_2$ fails, however, to have $C^0$-boundary at $q$ by Proposition \ref{Proposition variety cannot map to C0 boundary} but has $C^0$-boundary at all other points of $\partial W_2.$

\end{example}

\section{Proof of Corollary \ref{cor1ellipsoid}}\label{sectioncor1}

In this section, we give a proof of Corollary \ref{cor1ellipsoid}. Let $D$ and $\tau$ be as in Corollary \ref{cor1ellipsoid}.  First, it immediately follows from Theorem \ref{main theorem - negative constant case} that $\tau = -2(n + 1)^{-1}$. Before proving the remaining conclusion, we pause to recall the following simple fact. We will omit its proof, as it follows easily from the transformation law of the Bergman kernel.

\begin{lem}\label{Bergman kernel of the ellipsoid}
	Let $H = (h_{i\bar{j}})_{i, \bar{j} = 1}^n$ be a positive definite Hermitian matrix and $A$ the $n \times n$ matrix such that $H = AA^*$.  Let $E_H$ denote the complex ellipsoid given by \eqref{Ellipsoid equation}.
	The map $L: E_H \to \mathbb{B}^n$ given by $L(\zeta) = (n + 1)^{-{1 \over 2}}\zeta A$ is a biholomorphism and the Bergman kernel of $E_H$ is given by
	$$
	K_{E_H}(\zeta, \zeta) = {C \over (1 - {1 \over n + 1}\zeta H\zeta^*)^{n + 1}},
	$$
	where $\zeta$ is interpreted as a row vector and $C = {n! \over \pi^n}{\det H \over (n + 1)^n}$.
\end{lem}

We continue the proof of Corollary \ref{cor1ellipsoid}. We write $K=K_D$ for the Bergman kernel of $D$. We note that $K(z, p)$ has no zeros on $D$ by \cite[Theorem 1.2(1)]{DW22a} (this also follows from Theorem \ref{main theorem - negative constant case}).  Consequently, $T_p$ (given by $\eqref{Bergman representative coordinates}$) is holomorphic in $D$.
Letting $W \subset \mathbb{C}^n$ denote the image of $D$ under $T_p$, \cite[Theorem 1.2(2)]{DW22a} shows that $W \subset E_{g(p)}$.
Let $w_i = T_{p, i}$ and $w = T_p$ as in \eqref{Bergman representative coordinates}. By \cite[Theorem 1.2(3)]{DW22a} (with
$c^2 = 2(n + 1)^{-1}$, or by utilizing \cite[Lemma 3]{L65} instead) and Lemma \ref{Bergman kernel of the ellipsoid},
\begin{eqnarray*}
	\log\left({K(z, z)K(p, p) \over |K(z, p)|^2}\right) = \log\left( {1 \over 1 - {1 \over n + 1}\sum_{i, j = 1}^n w_i(z)g_{i\bar{j}}(p)w_{\bar{j}}(z)} \right)^{n + 1} = \log\left({K_{E_{g(p)}}(w(z), w(z)) \over C}\right), \quad z \in D,
\end{eqnarray*}
where $C$ is the constant in Lemma \ref{Bergman kernel of the ellipsoid}.  Taking $\partial\overline{\partial}$ of both sides, we find that
$$
\partial\overline{\partial} \log(K(z, z)) = \partial\overline{\partial} \log K_{E_{g(p)}}(w(z), w(z)),
$$
which shows that $T_p$ is a local isometry from
$D$ to $E_{g(p)}$ with respect to the
Bergman metrics 
of $D$ and $E_{g(p)}$ at every point of $D$. This implies that $W \subset E_{g(p)}$ is open, and thus a domain.  It remains to prove that $E_{g(p)}\setminus W$ is a Bergman-negligible subset of $E_{g(p)}$.

Let $L$ be the (nondegenerate) linear transformation in Lemma \ref{Bergman kernel of the ellipsoid}, which biholomorphically maps $E_{g(p)}$ to the unit ball $\mathbb{B}^n$. In particular, $L$ is also an isometry with respect to the Bergman metrics.  Set $H_p = L \circ T_p$ and $\hat{\Omega} = L(W) \subset \mathbb{B}^n$.  
Notice that $H_p$ is a holomorphic map from $D$ onto $\hat{\Omega},$ and it is also a local isometry with respect to the Bergman metrics $g$ and $g_{\mathbb{B}^n}$ of $D$ and $\mathbb{B}^n$.
On the other hand, by Theorem \ref{main theorem - negative constant case}, there is a biholomorphism $F$ from $D$ to $\Omega \subset \mathbb{B}^n$ where $\Omega$ is as described in Theorem \ref{main theorem - negative constant case}. In particular, $F$ is an isometry with respect to the Bergman metrics $g_{D}$ and $g_{\Omega}$ (and recall $g_{\Omega}= g_{\mathbb{B}^n}|_\Omega$).  Now both $H_p$ and $F$ are local isometries from $D$ into $\mathbb{B}^n$.  Thus, by a theorem of Calabi \cite{Ca53} (see also \cite[Theorem 1.1]{HL12}), there exists some $\phi \in \hbox{Aut}(\mathbb{B}^n)$ such that
\begin{equation}\label{consequence of Calabi rigidity theorem}
	\phi \circ F = H_p = L \circ T_p.
\end{equation}

Since the left hand side is one-to-one, so is the right hand side.  This implies that $T_p$ is one-to-one.  Therefore, $T_p: D \to W$ is a biholomorphism.
Moreover, by \eqref{consequence of Calabi rigidity theorem}, $\hat{\Omega} = \phi(\Omega)$.  By Theorem \ref{main theorem - negative constant case}, $\mathbb{B}^n \setminus \Omega$ is a Bergman-negligible subset of $\mathbb{B}^n$.  But every member of $\hbox{Aut}(\mathbb{B}^n)$, in particular $\phi$, extends to a biholomorphism in a neighborhood of $\overline{\mathbb{B}^n}$.  Consequently, it is clear that $\mathbb{B}^n \setminus \hat{\Omega}$ is also a Bergman-negligible subset of $\mathbb{B}^n$.  Finally, since $E_{g(p)}$ and $W$ differ from $\mathbb{B}^n$ and $\hat{\Omega}$ respectively by the linear transformation $L$, $E_{g(p)}\setminus W$ is a Bergman-negligible subset of $E_{g(p)}$. This finishes the proof of Corollary \ref{cor1ellipsoid}. \qed

\bibliographystyle{amsplain}

\bibliography{bibliography}

\fontsize{11}{9}\selectfont

\vspace{0.5cm}

\noindent pebenfelt@ucsd.edu;

 \vspace{0.2 cm}

\noindent Department of Mathematics, University of California San Diego, La Jolla, CA 92093, USA

\vspace{0.6 cm}

\noindent jtreuer@ucsd.edu

 \vspace{0.2 cm}

\noindent Department of Mathematics, University of California San Diego, La Jolla, CA 92093, USA

\vspace{.6 cm}

\noindent m3xiao@ucsd.edu;

 \vspace{0.2 cm}

\noindent Department of Mathematics, University of California San Diego, La Jolla, CA 92093, USA

\end{document}